\documentclass[reqno]{amsart}
\usepackage{amssymb}%, natbib}
\usepackage[pagebackref,nesting]{hyperref}
\usepackage[pdftex]{graphicx}

\newcommand{\R}{\mathbb{R}}
\newcommand{\N}{\mathbb{N}}

\newcommand{\E}{\mathbb{E}}
\newcommand{\PP}{\mathbb{P}}
\newcommand{\Eop}{\operatorname{\mathbb{\E}}}
\newcommand{\Pop}{\operatorname{\mathbb{\PP}}}
\newcommand{\F}{\mathbb{F}}
\newcommand{\A}{\mathcal{A}}
\newcommand{\K}{\mathcal{K}}

\newcommand{\I}{\mathcal{I}}

\newtheorem{theorem}{Theorem}
\newtheorem{corollary}[theorem]{Corollary}

\theoremstyle{definition}

\newtheorem{definition}[theorem]{Definition}
\numberwithin{equation}{section} \numberwithin{theorem}{section}

\def\0{\kern0pt\-\nobreak\hskip0pt\relax}

\makeatletter
\AtBeginDocument{%
 \def\@serieslogo{%
 \vbox to\headheight{%
 \parindent\z@ \fontsize{6}{7\p@}\selectfont
October 14, 2012\endgraf
 \vss}}}

\def\makeoverbar#1#2#3#4#5#6#7{%
 \setbox0=\hbox{$\m@th#2\mkern#5mu{{}#3{}}\mkern#6mu$}%
 \setbox1=\null \dimen@=#4\fontdimen8#13 \dimen@=3.5\dimen@
 \advance\dimen@ by \ht0 \dimen@=-#7\dimen@ \advance\dimen@ by \wd0
 \ht1=\ht0 \dp1=\dp0 \wd1=\dimen@
 \dimen@=\fontdimen8#13 \fontdimen8#13=#4\fontdimen8#13
 \rlap{\hbox to \wd0{$\m@th\hss#2{\overline{\box1}}\mkern#5mu$}}
 \fontdimen8#13=\dimen@}

\def\mylabel#1#2{{\def\@currentlabel{#2}\label{#1}}}

\makeatother

\begin{document}

\title[Stochastic Comparison and Applications to Control]{A Note on Applications of Stochastic Ordering to Control Problems in Insurance and Finance}

\author[N. \smash{B\"auerle}]{Nicole B\"auerle${}^*$}
\address[N. B\"auerle]{Department of Mathematics,
Karlsruhe Institute of Technology, D-76128 Karlsruhe, Germany}

\email{\href{mailto:baeuerle@kit.edu}
{baeuerle@kit.edu}}

%\urladdr{\href{http://www.mathematik.uni-karlsruhe.de/stoch/~baeuerle/}
%{http://www.mathematik.uni-karlsruhe.de/stoch/$\sim$baeuerle/}}

\thanks{Part of this research was carried out when the authors were visiting the Steklov Institute during the ``Stochastic Optimization and Optimal Stopping" workshop. We would like to thank the organizers for bringing us together. E. Bayraktar is supported in part by the National Science Foundation under
grants DMS 0906257, DMS 0955463, and DMS 1118673.}

\author[E. \smash{Bayraktar}]{Erhan Bayraktar$^\ddag$}
\address[E. Bayraktar]{Department of Mathematics,
University of Michigan, 530 Church Street, Ann Arbor, MI 48109, USA}

\email{\href{mailto:erhan@umich.edu} {erhan@umich.edu}}

\maketitle

\begin{abstract}
We consider a controlled diffusion process  $(X_t)_{t\ge 0}$  where the controller is allowed to choose the drift $\mu_t$ and the volatility $\sigma_t$ from a set $\K(x) \subset \R\times (0,\infty)$ when $X_t=x$. By choosing the largest $\frac{\mu}{\sigma^2}$  at every point in time an extremal process is constructed which is under suitable time changes stochastically larger than any other admissible process. This observation immediately leads to a very simple solution of problems where ruin or hitting probabilities have to be minimized. Under further conditions this extremal process also minimizes ``drawdown" probabilities.
\end{abstract}

\vspace{0.5cm}
\begin{minipage}{14cm}
{\small
\begin{description}
\item[\rm \textsc{ Key words}]
{\small Time changed continuous Martingale, Stochastic Ordering,}\\
{\small  Ruin Problem.}
\item[\rm \textsc{AMS subject classifications}]
{\small  60H30, 91B30. }
\end{description}
}
\end{minipage}

\section{Introduction}\label{sec:intro}\noindent
Consider a one-dimensional controlled diffusion process  $(X_t)_{t\ge 0}$  where the controller is allowed to choose the drift $\mu_t$ and the volatility $\sigma_t$ from a set $\K(x) \subset \R\times (0,\infty)$ when $X_t=x$. Note that such a process is not necessarily Markovian. We are essentially interested in the problem of minimizing or maximizing the probability that the process hits a certain barrier $b\in\R$. If $b=0$ this can be interpreted as  ruin. Problems like this often occur in continuous-time gambling problems or in actuarial or financial applications where one wants to avoid ruin. In \cite{MR812818} the authors, among others, consider in the setting of continuous-time gambling, the problem of reaching the target $1$ with maximal probability. They show that maximizing the ratio $\frac{\mu}{\sigma^2}$ at all time points solves the problem. The ratio $\frac{\mu}{\sigma^2}$ often occurs in discrete-time gambling and serves as a measure of superfairness (see e.g. \cite{ds}). The proof of this result given in  \cite{MR812818} follows standard techniques and consists of guessing the optimal control and then verifying that the associated value $Q$ satisfies $Q\ge V$  where $V$ is the value function. This procedure involves a number of analytic results.

We instead build on the following simple observation: When optimization criteria like hitting probabilities are considered, time transformations of the processes do not change such probabilities. Moreover for different processes which are admissible, it is possible to use different time changes. Finally constructing the time change in such a way that the diffusive part of the process is always the same, leads to a favorable situation. Indeed we will show in Theorem \ref{theo:comparisonviatimechange} that the process which is obtained by maximizing the ratio $\frac{\mu}{\sigma^2}$ at all time points (name this the extremal process) is under suitable time changes stochastically larger than any other admissible process. This is a very strong result and the proof we use here is very simple and based on the time change property of the Brownian motion.
It is worth mentioning that in \cite{hajek}, part III also a time change is used to compare processes, however the time change is different and applied to a zero drift process. The main aim in \cite{hajek} is to construct extremal processes w.r.t.\ the increasing convex ordering by dominating one process by convex combination of other processes.
For a recent paper on the stochastic comparison of Markov processes see \cite{ruwo11}. However note that our processes are not necessarily Markovian and allow a comparison only after a time change.

In case the maximal ratio $\frac{\mu}{\sigma^2}$ does not depend on $x$ we can even strengthen our result to showing that the extremal process maximizes $(X_t-\alpha M_t)_{t\ge 0}$ stochastically under suitable time changes where $\alpha\in[0,1]$ and $M_t$ is the running maximum of the process. Thus the probability of a 'drawdown' is minimized by the extremal  process.

Our paper is organized as follows: In the next section we introduce the basic model, the time change and our two main results. In Section \ref{sec:specialmodels} we consider a few special cases and in Section \ref{sec:ruin} we consider some ruin problems. Part of them have already been solved by stochastic control methods but we show here that our approach is much simpler.

\section{The Model}\label{sec:mod}\noindent
We consider a one-dimensional controlled diffusion process $(X_t)_{t\ge 0}$ on a filtered probability space $(\Omega, \F, (\F_t)_{t\ge 0},\Pop)$ with
\begin{equation}\label{eq:sde}
dX_t = \mu_t dt + \sigma_t d W_t,\quad X_0=x\in\R
\end{equation}
where $(W_t)_{t\ge 0}$ is an $\F$-Brownian motion. We assume that the drift $\mu_t$ and the volatility $\sigma_t$ can be chosen by a controller from a set $\K(x) \subset \R\times (0,\infty)$ when $X_t=x$. Thus, the set $\K(x)$ gives all admissible pairs of drift and volatility when the process is in state $x\in\R$. The drift  $(\mu_t)$ and volatility $(\sigma_t)$ are supposed to be $\F$-progressively measurable and satisfy
\begin{equation}\label{eq:integrability}
\int_0^t (|\mu_u|+\sigma^2_u) du <\infty\quad \Pop-a.s.
\end{equation}
for all $t\ge 0$.  Let us denote by $\A(x)$ the set of all processes that can be constructed in this way with $X_0=x$.

Suppose now $(X_t)_{t\ge 0}\in\A(x)$ is an arbitrary process. Obviously
$M_t := \int_0^t \sigma_u dW_u$ is a continuous local martingale. The quadratic variation of $(M_t)_{t\ge 0}$ is given by
$$ \langle M\rangle_t = \int_0^t \sigma^2_u du.$$
Define $T(t) := \inf\{u\ge 0\; |\; \langle M\rangle_u >t\}$. Then $M_{T(t)}$ is
equal to a Brownian motion (see e.g.\ \cite{ks}, Theorem 3.4.6) which we denote again by $(W_t)_{t\ge 0}$ for simplicity. We
introduce the following {\em time change}:
$$ X_{T(t)} = x+ \int_0^{T(t)}\mu_u du + M_{T(t)} = x + \int_0^{T(t)}\mu_u du +
W_t.$$ Changing the integration variable in the second term we obtain
\begin{eqnarray*}
% \nonumber to remove numbering (before each equation)
  \int_0^{T(t)}\mu_u du &=& \int_0^t \mu_{T(u)} T'(u)du.
\end{eqnarray*}
Since $T(\langle M\rangle_t)=t$ we get:
\begin{equation}\label{eq:timechange}
X_{T(t)} = x+ \int_0^t
\frac{\mu_{T(u)}}{\sigma^2_{T(u)}} du+W_t.
\end{equation}

We assume now further that there exist measurable functions $m(\cdot)$ and $s(\cdot)>0$ such that
\begin{equation}\label{eq:maxpoints}
{m(x) \over s^2(x)}=\sup\left\{{\mu \over \sigma^2}: (\mu,\sigma) \in \K(x)\right\}, \quad x \in \R,
\end{equation}
and we denote this special process by
\begin{equation}\label{eq:SDE-extremal}
dX_t^* = m(X_t^*) dt + s(X_t^*) dW_t.
\end{equation}
In what follows we will assume that \eqref{eq:SDE-extremal} has a unique weak solution. This is for example guaranteed by
\[
\int_{x-\varepsilon}^{x+\varepsilon} \frac{1+|m(y)|}{s^{2}(y)}dy<\infty, \quad \text{for some $\varepsilon>0$}
\]
for all $x \in \mathbb{R}$. Recall that we already assumed that $s>0$.

In what follows we will show that $(X^*_t)_{t\ge 0}$ has some extremal properties in the class $\A(x)$. To this end, consider the following definition:

\begin{definition}\label{def:st}
For two stochastic processes $(X_t)_{t\ge 0}$ and $(Y_t)_{t\ge 0}$ we say that $(X_t)_{t\ge 0}$ is {\em stochastically smaller} than $(Y_t)_{t\ge 0}$ (written $(X_t) \le_{st} (Y_t)$) if
$$ \Eop h(X_{t_1},\ldots ,X_{t_n}) \le \Eop h(Y_{t_1},\ldots ,Y_{t_n})$$ for all $0\le t_1\le\ldots\le t_n, n\in\N$ and measurable functions $h:\R^n\to\R$ which are increasing in all components and are such that the expectations exist.
\end{definition}

Recall that two real-valued random variables $X$ and $Y$ are ordered by $X\le_{st} Y$ when $\Eop h(X)\le \Eop h(Y)$ for all $h:\R\to\R$ increasing where the expectations exist. The definition of $(X_t) \le _{st} (Y_t)$ given above is one of a number of equivalent statements: When we denote by $C[0,\infty)$ the set of continuous functions on $[0,\infty)$ and introduce the following partial ordering for $f,g\in C[0,\infty)$:
$$ f<_c g\quad\mbox{if}\; f(t) \le g(t), \forall t\ge 0,$$
Then  $(X_t) \le _{st} (Y_t)$ is equivalent to $$\Eop h\big((X_t)_{t\ge 0} \big) \le \Eop h\big((Y_t)_{t\ge 0} \big)$$
for all $h$ such that $h:C[0,\infty) \to\R$, $h$ is measurable, increasing w.r.t. $<_c$ and the expectations exist (see e.g. Theorem 4.1.1 in \cite{stoyan}). Now using Strassen's theorem (see e.g. \cite{szekli} Theorem 2.8.A), we know that $(X_t) \le _{st} (Y_t)$ is also equivalent to the existence of a common probability space $(\Omega,\mathcal{F},\Pop)$ and processes $(\hat{X}_t)_{t\ge 0}$ and  $(\hat{Y}_t)_{t\ge 0}$ on it such that $ (\hat{X}_t)_{t\ge 0} \stackrel{d}{=}  ({X}_t)_{t\ge 0}, $  $(\hat{Y}_t)_{t\ge 0}\stackrel{d}{=} (Y_t)_{t\ge 0}$ and $X_\cdot(\omega)  <_c Y_\cdot (\omega)$ for all $\omega\in\Omega$.

We obtain the following result:

\begin{theorem}\label{theo:comparisonviatimechange}
Suppose $(X_t^*)_{t\ge 0}\in \A(x)$ is defined as in \eqref{eq:SDE-extremal} and $(X_t)_{t\ge 0}\in \A(x)$ is arbitrary. Assume further that the function $m/s^2$ is locally Lipschitz and has linear growth.
Then there exist continuous and strictly increasing time changes $T, U:\R_+\to\R_+$ such that
$$ \big(X_{T(t)}\big) \le_{st} \big(X^*_{U(t)}\big).$$
\end{theorem}

\begin{proof}
We choose time change processes $U$ and $T$ such that $U(\int_0^t s(X_u^*)^2du)=t$ and $T(\int_0^t \sigma_u^2du)=t$. Then
$$ X^*_{U(t)} = x + \int_0^t {m\big(X^*_{U(u)}\big) \over s^2\big( X^*_{U(u)}\big)} du + W_t^*$$
and
$$ X_{T(t)} = x + \int_0^t {\mu_{T(u)} \over \sigma^2_{T(u)}} du + W_t$$
or when we use the definition $Y^*_t := X^*_{U(t)}$ and  $Y_t := X_{T(t)}$ we obtain:
\begin{eqnarray*}
% \nonumber to remove numbering (before each equation)
  Y_t^* &=&  x + \int_0^t {m\big(Y^*_u\big) \over s^2\big( Y^*_{u}\big)} du + W_t^* \\
  Y_t &=&  x + \int_0^t {\mu_{T(u)} \over \sigma^2_{T(u)}} du + W_t.
\end{eqnarray*}
Since $m/s^2$ is locally Lipschitz and has linear growth, the SDE
\[
\hat{Y}_t=x + \int_0^t {m\big(\hat{Y}_u\big) \over s^2\big( \hat{Y}_{u}\big)} du + W_t,
\]
has a unique strong solution. Moreover, its law coincides with that of $Y^*$.
The advantage of working with $\hat{Y}$ is that, it is driven by the same Brownian motion $W$ that drives $Y$. We will show that $\hat{Y}$ dominates $Y$ path-wise, which will directly imply the statement of the theorem.
Let
\[
\tau_{N}:=\inf\{t \geq 0: \max\{|Y_t|, |\hat{Y}_t|\} \geq N\}.
\]
Since the difference $\hat{Y}-Y$ is of finite variation we have
\[
\begin{split}
(Y_{t \wedge \tau_N}-\hat{Y}_{t \wedge \tau_N})^{+}&\le\int_{0}^{t \wedge \tau_N}1_{\{Y_s>\hat{Y}_s\}}\left({\mu_{T(s)} \over \sigma^2_{T(s)}}- {m\big(\hat{Y}_s\big) \over s^2\big( \hat{Y}_{s}\big)}\right)ds
\\& \leq \int_{0}^{t \wedge \tau_N}1_{\{Y_s>\hat{Y}_s\}}\left({m\big(Y_s\big) \over s^2\big( Y_{s}\big)}- {m\big(\hat{Y}_s\big) \over s^2\big( \hat{Y}_{s}\big)}\right)ds
\\&\leq L_N \int_0^{t \wedge \tau_N}(Y_s-\hat{Y}_s)^+ds,
\end{split}
\]
which implies that $\hat{Y}_{t \wedge \tau_N} \geq Y_{t \wedge \tau_N}$ due to an application of Gronwall's lemma. Here, we first use Jensen's inequlity, the second inequality comes from \eqref{eq:maxpoints}, and $L_N$ is a Lipschitz constant. Now the result follows since $t \wedge \tau_N =t$ for large enough $N$.
\end{proof}

\begin{corollary}\label{cor:inf}
Suppose $(X_t^*)_{t\ge 0}\in \A(x)$ is defined as in \eqref{eq:SDE-extremal} and $(X_t)_{t\ge 0}\in \A(x)$ is arbitrary. Then under the assumptions of Theorem \eqref{theo:comparisonviatimechange} we obtain:
$$ \inf_{t\ge 0} X_t \le_{st} \inf_{t\ge 0}X_t^*.$$
\end{corollary}

\begin{proof}
Since $\inf_{t\ge 0} X_t = \inf_{s\ge 0} X_{T(s)}$ and $g\circ h\big((x_t)_{t\ge 0}\big) = g(\inf_{t\ge 0}x_t)$ is increasing w.r.t. $<_c$ for $g:\R\to\R$ increasing, the statement follows from Theorem \ref{theo:comparisonviatimechange} and the discussion after Definition \ref{def:st}.
\end{proof}

For the next result suppose that $\K(x)\equiv \K$ does not depend on $x$. Then
$$ {m\over s^2} = {m(x) \over s^2(x)},\quad\mbox{for all}\; x\in\R.$$
Let us define $M_t := \sup_{0\le s\le t} X_s$ and $M_t^*=\sup_{0\le s\le t} X_s^*$. Then we obtain that the extremal process $X^*$ satisfies:
\begin{theorem}\label{thm:drowdown}
Suppose that $\K(x)\equiv \K$ for all $x\in\R$ and $(X_t^*)_{t\ge 0}\in \A(x)$ is defined as in \eqref{eq:SDE-extremal} and $(X_t)_{t\ge 0}\in \A(x)$ is arbitrary. Moreover, assume that the initial state $x$ is non-negative. Then it holds for $\alpha\in [0,1]$
$$\Big( X_{T(t)}-\alpha M_{T(t)}\Big) \le_{st} \Big( X^*_{U(t)}-\alpha M^*_{U(t)}\Big),$$
where $T, U:\R_+\to\R_+$ are the time changes from Theorem \ref{theo:comparisonviatimechange}. \end{theorem}

\begin{proof}
First note that $(X_t)$ and $(X_t^*)$ can be constructed on a common probability space without
changing their distribution such that the driving Brownian motion is the same (as in the proof of Theorem~\ref{theo:comparisonviatimechange}). Using this construction we obtain:
$$ X^*_{U(t)}-X_{T(t)} \stackrel{d}{=} \int_0^t \Big(\frac{m}{s^2}-{m(X_{T(u)}) \over \sigma^2(X_{T(u)})}\Big) du.$$
By our assumption ${m\over s^2} \ge {\mu(x) \over \sigma^2(x)}$ for all $x\in\R$ hence $X^*_{U(t)}\ge X_{T(t)}$ $\Pop$-a.s. for all $t >0$ and
$$X^*_{U(t)}-X_{T(t)} \uparrow \;\;\mbox{in}\; t \;\Pop-a.s.$$
Let $t>0$ be arbitrary and define $\tau:=\sup\{s<t : X^*_{U(s)} = M^*_{U(t)}\}$.  It holds (note that $M_{T(t)}\ge 0$ since $X_0\ge 0$)
\begin{eqnarray*}
% \nonumber to remove numbering (before each equation)
  X_{T(t)} -\alpha M_{T(t)} &=& X_{T(t)} - X_{T(\tau)}
  +X_{T(\tau)} -\alpha M_{T(t)} \\
   &\le_{st}& X^*_{U(t)} - X^*_{U(\tau)}
  +X_{T(\tau)}(1 -\alpha) \\
   &\le_{st}& X^*_{U(t)} - X^*_{U(\tau)}  +X^*_{U(\tau)}(1 -\alpha)\\
  &=& X^*_{U(t)} - \alpha X^*_{U(\tau)} = X^*_{U(t)}
  -\alpha M^*_{U(t)}.
\end{eqnarray*}
This inequality holds also true when we consider jointly different time points $0\le t_1< t_2<\ldots <t_n$ and thus the statement follows.
\end{proof}

As an application we have that the extremal $X^*$ minimizes the probability of an $\alpha$ drawdown, which we define to be the event $X_t \leq \alpha M_t$ for some $t\in\R$.

\begin{corollary}
Under the assumptions of Theorem~\ref{thm:drowdown} we have that $X^*$ in \eqref{eq:SDE-extremal} minimizes the probability
\[
\mathbb{P}(X_t \leq \alpha M_t\; \mbox{for some}\; t\in\R), \quad \text{for $X \in \mathcal{A}(x)$},
\]
for any $\alpha \in (0,1)$.
\end{corollary}
The above criteria could be useful for fund managers who do offer this type of guarantee in order to
satisfy the aversion to deception of the investors. (For relevant references and an alternative optimization problem in which the drawdown is taken to be a constraint, see \cite{MR2410840}.)

\section{Some Special Problems}\label{sec:specialmodels}
\subsection{Controlled Diffusions}
A special case of the model in Section \ref{sec:mod} is a controlled It\^{o}-process given by
$$X_t = x+  \int_0^t \mu(X_s,u_s) ds + \int_0^t\sigma(X_s,u_s) d W_s,\quad t\ge 0$$
where $u=(u_t)_{t\ge 0}$ is an $\F$-adapted control process taking values in a set $U\subset \R^m$ and
$\mu:\R\times U\to \R$ and $\sigma:\R\times U\to \R$ are continuous
functions and satisfy the usual growth conditions which imply that
the stochastic differential equation has a solution. Here the extremal process is obtained by
\begin{equation}\label{eq:maxpoints2}
{\mu(x,u^*) \over \sigma^2(x,u^*)}=\sup\left\{{\mu(x,u) \over \sigma^2(x,u)}: u\in U\right\}, \quad x \in \R.
\end{equation}

\subsection{Switching Problems}
Another special case of the model presented in Section \ref{sec:mod} are the so-called switching problems. Here we suppose that we have a set of possible regimes $\I =\{1,\ldots,m\}$ and a switching control consists of a sequence $(\tau_n,i_n)_{n\in\N}$ where $(\tau_n)$ is an increasing sequence of $\F$-stopping times and $i_n$ are $\F_{\tau_n}$-measurable, $\I$-valued random variables, representing a regime. Given such a sequence we define the control process
$$ I_t := \sum_{n\ge 0} i_n 1_{[\tau_n,\tau_{n+1})}(t),\quad t\ge 0.$$
The controlled process is then given by
$$ dX_t = \mu(X_t,I_t) dt + \sigma(X_t,I_t)dW_t,\quad t\ge0$$
for suitable functions $\mu_i(\cdot)=\mu(\cdot,i)$ and $\sigma_i(\cdot)=\sigma(\cdot,i)$. Let $$R_i := \left\{x\in\R : {\mu(x,i) \over \sigma^2(x,i)}=\max\left\{{\mu(x,j) \over \sigma^2(x,j)}: j \in \I\right\}\right\}.$$
The extremal process is here obtained by choosing regime $i\in\I$ if and only if the process is currently in set $R_i$.

\section{Minimizing the Probability of Ruin}\label{sec:ruin}
In this section we consider some more specific applications of our previous findings.
Corollary \ref{cor:inf} in particular implies that the process $X^*$ minimizes the probability $\Pop_x(X_t < b\;\mbox{for some time point}\: t\in\R)=\Pop(\inf_{t\ge 0} X_t < b)$ for arbitrary $b\in\R$ whenever this is a reasonable and non-trivial problem.

\begin{theorem}\label{thm:prb-o-ruin}
The extremal process $X^*$ (which does not depend on $b$) minimizes the probability of ruin. The minimum probability of ruin is given by
\[
\inf_{X \in \mathcal{A}(x)}\Pop_x(\inf_{t\ge 0} X_t < b)=1-\frac{p(x)}{p(\infty)},
\]
where $p$ is the scale function of $X^*$, i.e.,
\[
p(x)=\int_{b}^{x}\exp\left[-2\int_b^{\xi}{m(u) \over s^2(u)}du\right]d \xi.
\]
\end{theorem}
\begin{proof}
The proof that $X^*$ minimizes the probability of ruin follows from Corollary \ref{cor:inf} and the fact that $h(x)=1-1_{\{x < b\}}$ is an increasing function. The expression for the minimum probability of ruin is then simply the probability that $X^*$ does not reach $b$, which can be easily computed in terms of the scale function as in page 344 of \cite{ks}.
\end{proof}

In particular if $(X_t)$ is the surplus process of an insurance company and $b=0$ we are faced with a so-called problem of minimizing the ruin problem which we will analyze in the next sections. We observe the underlying optimality principle of seemingly different problems are the same and are given by the extremal process $X^*$. In the stochastic control approach that was taken in the literature each case would have to be solved and then \emph{verified} separately.

\subsection{Controlling Risk Processes by Proportional Reinsurance}\label{sec:propreinsurance}
%\begin{example}
In Section 2.2.1 of \cite{schm08}  one can find a proportional reinsurance
model where the surplus process is given by
$$ X_t = x_0 + \int_0^t \big( u_s \theta-(\theta-\eta)\big) ds +
\sigma\int_0^t u_s d W_s$$ and $u_t\in[0,1]$ denotes the retention level at time $t$ of a proportional reinsurance contract.  In this case we have
$$ \mu(u)=u\theta-(\theta-\eta),\quad \sigma(u) = u \sigma$$
and $U=[0,1]$. We assume that $\eta<\theta$.

We divide the problem into two cases. First, we assume that $u>\varepsilon$ for small $\varepsilon>0$ and apply Theorem~\ref{thm:prb-o-ruin}. With this restriction we know in order to minimize the probability of ruin we have to maximize the expression
$$ \frac{\mu(u)}{\sigma^2(u)} = \frac{u\theta -\theta+\eta}{u^2
\sigma},\quad u\in(\varepsilon,1].$$  This gives us $$u^* =
2(1-\frac{\eta}\theta)\wedge1.$$

When $u_t \leq \varepsilon$, for some $t$, changing the control at that time to $u_t=(\theta-\eta)/\theta$, one could build a process that path-wise dominates the original controlled process. Therefore, without loss of generality we can assume that $u>\varepsilon$ and, hence, the control $u_t \equiv u^*$
minimizes the probability of ruin in this case. Compare with Lemma 2.4 in \cite{schm08}, in which the problem has been solved by constructing the solution explicitly and then using a verification theorem.

\subsection{Controlling Risk Processes by XL Reinsurance}
Alternatively one could try and find the excess-of-loss reinsurance which minimizes the probability of ruin. In such a reinsurance contract a limit $u\ge 0$ is chosen such that the insurance company has to pay $\min\{Z,u\}$ when $Z$ is a claim. In this case it is reasonable to define the surplus process by
$$ X_t = x_0 + \int_0^t \big( \theta \mu(u_s)-(\theta-\eta)\Eop Z\big) ds +
\sigma\int_0^t \sigma(u_s) d W_s$$ and $u_t\ge 0$ denotes the limit at time $t$ of an excess-of-loss reinsurance contract and $$ \mu(u) := \Eop\min\{Z,u\},\quad \sigma^2(u) := \Eop \min\{Z,u\}^2.$$ Again we assume $\eta<\theta$. Proceeding like in the previous subsection we have to maximize
$$ \frac{\theta \mu(u)-(\theta-\eta)\Eop Z}{\lambda^2 \sigma^2(u)}$$ over $u\in U=[0,\infty)$ (the case $u=0$ can be treated as in subsection \ref{sec:propreinsurance}). If $Z$ is exponentially distributed with parameter $\lambda>0$ then we have to maximize
$$ \frac{\frac\theta\lambda (1-e^{-\lambda u})-(\theta-\eta)\frac1\lambda}{2-e^{-\lambda u}(2+2u\lambda -u^2 \lambda^2(\lambda-1))}$$ over $u\ge 0$. For $\theta=2, \lambda=\eta=1$ the solution is given by $u^* = 2+L(-2e^{-2})$ where $L(x)$ is the upper branch of Lambert's $W$-function which is the inverse of $f(x)=xe^x$. To the best of our knowledge this problem has not been treated before.

\subsection{Controlling Risk Processes by Investment}
The next problem which has been considered in \cite{schm08}, Section 2.2.2 is the problem of optimal investment
for an insurance company. Here the surplus process is given by
$$dS_t = \eta dt +\sigma_S dW_t^S$$
and the insurer can invest into a risky asset given by
$$dZ_t = mZ_t dt+\sigma_I Z_t dW_t^I.$$
The Brownian motions $(W_t^S)_{t\ge 0}$ and $(W_t^I)_{t\ge 0}$ are supposed to be independent. Given the investment strategy $(u_t)_{t\ge 0}$, where $u_t\in\R$ is the amount invested in the risky asset at time $t$ we arrive at the controlled surplus process
$$dX_t = (\eta+u_t m)dt + \sigma_S dW_t^S +\sigma_I u_t dW_t^I.$$
This process is in distribution equal to the process
$$dX_t = (\eta+u_t m)dt + \sqrt{\sigma_S^2+\sigma_I^2 u_t^2} dW_t$$
with one Brownian motion $(W_t)$. When we want to minimize the probability of ruin in this case we have to maximize
$$ \frac{\eta+um}{\sigma_S^2+\sigma_I^2 u^2}$$ over all $u\in\R$. This yields
$$u^*\equiv \frac{1}{m\sigma_I} \Big(\sqrt{\eta^2 \sigma_I^2+m^2\sigma_S^2}-\eta \sigma_I \Big)$$
compare Theorem 2.7 in \cite{schm08}

\subsection{Controlling Risk Processes by Reinsurance and Investment}
In section 2.2.3 of \cite{schm08} the combination of both investment and reinsurance is considered. The controlled process is here given by
$$dX_t = (b_t\theta-(\theta-\eta)+mA_t)dt + \sqrt{\sigma_S^2b_t^2+\sigma_I^2 A_t^2}dW_t$$
with control $u_t=(b_t,A_t)$ and assumption $\eta<\theta\le\eta+\sqrt{\eta^2+m^2\sigma_S^2/\sigma^2}$. In order to minimize the probability of ruin in this case we have to maximize
\begin{equation} \label{eq:ratio}
\frac{b\theta-(\theta-\eta)+mA}{\sigma_S^2b^2+\sigma_I^2 A^2}
\end{equation}
over $A\in\R$ and $b\in[0,1]$ which gives us
\begin{equation}\label{eq:optimal}
A^*=\frac{2m\sigma_S^2(\theta-\eta)}{m^2\sigma_S^2+\theta^2\sigma_I^2},\quad b^*= \frac{2\theta\sigma_I^2 (\theta-\eta)}{m^2\sigma_S^2+\theta^2\sigma_I^2}\wedge 1,
\end{equation}
compare Theorem 2.8 in \cite{schm08}. The case $b=A=0$ can be treated as in subsection \ref{sec:propreinsurance}.

In \cite{MR2200024} a slightly different approach was taken.
Inspired by the results of \cite{MR812818}, the authors in \cite{MR2200024} first guess the optimal portfolio by minimizing the ratio \eqref{eq:ratio}, use the form of the portfolio to construct the value function explicitly and then use a verification theorem to prove optimality. We, on the other hand, have directly (using the stochastic ordering arguments) shown that the strategy in \eqref{eq:optimal} is optimal. We also have shown the optimality of the extremal process $X^*$ for general diffusion processes and calculated the corresponding minimum probability of ruin. As an aside it also turns out that the same strategies minimizes the drawdown.
%\end{example}

\subsection{Controlling Investment into a Financial Market}
 Let $(S_1,\cdots,S_n)$ be the adapted price processes of risky assets satisfying
\[
{dS_i(t) \over S_i(t)}=\mu_idt+\sum_{\nu=1}^{n}\sigma_{i \nu}
dW_{\nu}(t),
\]
with $S_i(0)=s_i>0$. We will denote the covariance matrix by
$a:=\sigma \sigma'$ which we will assume to be positive definite.
%(The drift and the volatilities are adapted to
%the filtration of the underlying Brownian motion.)

Let $(\pi_t)_{t\ge 0}=(\pi_1(t),\cdots \pi_n(t))_{t\ge 0}$ denote an adapted
portfolio process where $\pi_i(t)$ gives the proportion of the wealth which is invested in asset $i$ at time $t$, i.e., $\sum_{i=1}^{n} \pi_i(t)=1$ for all $t\ge 0$. Then the wealth satisfies
\[
{dX^{\pi}(t) \over X^{\pi}(t)}=\sum_{i=1}^{n}\pi_i(t){d S_i(t) \over
S_i(t)}=\mu^{\pi}(t)dt+\sum_{\nu=1}^{n}\sigma_{\nu}^{\pi}(t)dW_{\nu}(t)
\]
with $X^{\pi}(0)=1$. Here,
\[
\mu^{\pi}(t):=\sum_{i=1}^{n}\pi_i(t)\mu_i, \quad
\sigma_{\nu}^{\pi}:=\sum_{i=1}^{n} \pi_i(t) \sigma_{i \nu}
\]
Solution to this equation is
\[
\log
X^{\pi}(t)=\int_0^{t}\gamma^{\pi}(s)ds+\sum_{\nu=1}^{n}\int_0^{T}\sigma_{\nu}^{\pi}(s)dW_{\nu}(s)
\]
where \[ \gamma^\pi(t) = \mu^\pi(t)-\frac12 a^\pi(t).\] Let us
denote $a^{\pi \pi}=\sum_{\nu=1}^{d}(\sigma_{\nu}^{\pi})^2.$ Obviously for $a>0$ the probability
$$\Pop(X_t^\pi > a\; \mbox{for all}\; t\ge 0) = \Pop(\log X_t^\pi > \log(a)\; \mbox{for all}\; t\ge 0).$$
Now, thanks to Theorem~\ref{thm:prb-o-ruin} when the drift and the
volatility coefficients are constant then the portfolio that
minimizes the probability of ruin is given by the maximizer of
\[
{\gamma^{\pi} \over a^{\pi \pi}}={\sum_{i=1}^n \pi_i \mu_i \over
\sum_{i=1}^n \sum_{j=1}^n \pi_i a_{ij}\pi_j}-{1 \over 2},
\]
subject to the constraint that $\sum_{i=1}^{n}\pi_i=1$, instead of the maximizer of the so-called Sharpe ratio (the ratio of mean value of the portfolio to its standard deviation) in the mean-variance portfolio analysis of Markowitz
or the maximizer of the growth rate (which is the numerator of the above expression) which yields the log/growth optimal portfolio whose maximizer is given by the so-called Merton ratio. For further comparisons of optimizers of different optimality criteria, see
Section 4 of \cite{Fernholz-Karatzas-survey}.
 %or the maximizer of the Merton ratio (the ratio of mean value of the portfolio to its variance) which appears in growth optimal portfolios.

\bibliography{bib-ruin}

\begin{thebibliography}{10}

\bibitem{ds}
{\sc L.~E. Dubins and L.~J. Savage}, {\em How to gamble if you must.
  {I}nequalities for stochastic processes}, McGraw-Hill Book Co., New York,
  1965.

\bibitem{MR2410840}
{\sc R.~Elie and N.~Touzi}, {\em Optimal lifetime consumption and investment
  under a drawdown constraint}, Finance Stoch., 12 (2008), pp.~299--330.

\bibitem{Fernholz-Karatzas-survey}
{\sc E.~R. Fernholz and I.~Karatzas}, {\em Stochastic {P}ortfolio {T}heory: {A}
  {S}urvey}, Handbook of Numerical Analysis, 15 (2009), pp.~89--168.

\bibitem{hajek}
{\sc B.~Hajek}, {\em Mean stochastic comparison of diffusions}, Z. Wahrsch.
  Verw. Gebiete, 68 (1985), pp.~315--329.

\bibitem{ks}
{\sc I.~Karatzas and S.~E. Shreve}, {\em Brownian motion and stochastic
  calculus}, vol.~113 of Graduate Texts in Mathematics, Springer-Verlag, New
  York, second~ed., 1991.

\bibitem{MR812818}
{\sc V.~C. Pestien and W.~D. Sudderth}, {\em Continuous-time red and black: how
  to control a diffusion to a goal}, Math. Oper. Res., 10 (1985), pp.~599--611.

\bibitem{MR2200024}
{\sc S.~D. Promislow and V.~R. Young}, {\em Minimizing the probability of ruin
  when claims follow {B}rownian motion with drift}, N. Am. Actuar. J., 9
  (2005), pp.~109--128.

\bibitem{ruwo11}
{\sc L.~R\"uschendorf and V.~Wolf}, {\em Comparison of markov processes via
  infinitesimal generators}, Statistics \& Decision, 28 (2011), pp.~151--168.

\bibitem{schm08}
{\sc H.~Schmidli}, {\em Stochastic control in insurance}, Probability and its
  Applications (New York), Springer-Verlag London Ltd., London, 2008.

\bibitem{stoyan}
{\sc D.~Stoyan}, {\em Comparison methods for queues and other stochastic
  models}, John Wiley \& Sons Ltd., Chichester, 1983.

\bibitem{szekli}
{\sc R.~Szekli}, {\em Stochastic ordering and dependence in applied
  probability}, Springer-Verlag, New York, 1995.

\end{thebibliography}
\bibliographystyle{siam}

\end{document}